\documentclass{tac}

\usepackage{amsfonts}
\usepackage{tikz}
\usepackage{bbm}
\usepackage{amsmath}
\usepackage{amssymb}

\usepackage[all,cmtip,2cell]{xy}
\usepackage{tensor}
\usepackage{adjustbox}
\usepackage{amssymb}
\usepackage{xfrac}
\usepackage{faktor}
\UseAllTwocells
\usepackage{pdflscape} % Allows certain pages to be printed in landscape

\usepackage{xy}
\input diagxy

\usepackage[colorlinks=true]{hyperref}
\hypersetup{allcolors=[rgb]{0.1,0.1,0.4}}

\author{Phillip M Bressie and David N Yetter}

\thanks{Substantial progress on this work was made while both authors were visiting the Mathematical  Sciences  Research  Institute,  Berkeley,  CA  and  the  balance  while  the  second author was either in residence at or attached to MSRI to the extent allowed by the anti-COVID-19  social  distancing  regime  begun  in  March  2020.   The  authors  wish  to  thank MSRI for its hospitality and the National Science Foundation for financial support under grant DMS-0441170.}

\address{Spring Hill College, Mobile, AL\\
Kansas State University, Manhattan, KS}

\title{Weakly $\omega$-categorified models of algebraic theories}

\copyrightyear{2020}


\eaddress{pbressie@shc.edu\CR dyetter@ksu.edu}

\newcommand{\catname}[1]{{\normalfont\textbf{#1}}}
\newcommand{\Set}{\catname{Set}}

\newcommand{\Cat}{\catname{Cat}}
\newcommand{\Graph}{\catname{Graph}}
\newcommand{\Glob}{\catname{Glob}}

\newcommand{\GlobPRO}{\catname{GlobPRO}}

\newcommand{\CartDuoidal}{\catname{CartDuoidal}}

\newcommand{\Coll}{\catname{Coll}} %Given two l's since \Col is apparently defined elsewhere

\newcommand{\Cont}{\catname{Cont}}

\newcommand{\wkp}{{\mathcal{W}}{\mathcal{P}}}
\newcommand{\ip}{{\mathfrak S}{\mathcal{P}}}
\newcommand{\Mon}{\catname{Mon}}
\newcommand{\one}{\textbf{1}}
\newcommand{\Tone}{\mathcal{T}(\one)}
\newcommand{\Tsquared}{\mathcal{T}^2(\one)}
\newcommand{\id}{\mathbbm{1}}
\DeclareSymbolFont{AMSa}{U}{msa}{m}{n}
\DeclareMathSymbol{\mysquare}{\mathord}{AMSa}{"03}
\newcommand{\pullbackmark}[2]{\save ;p+<.8pc,0pc>:(0,-1)::%
	(#1) *{\phantom{Z}} %
	;p+(#2)-(0,0) **@{-}%
	;p-(#1)+(0,0) *{\phantom{Z}} **@{-} \restore}
\newcommand{\pullbackSub}[2]{{}_{#1}\kern-\scriptspace{\times}_{#2}}
\newtheorem{theorem}{Theorem}

\newtheoremrm{rem}{Remark}

\mathrmdef{Hom}
\mathbfdef{Set}

\begin{document}

\maketitle
\begin{abstract}
 We provide the expected constructions of weakly $\omega$-categorified models (in the sense of Bressie) of the theory of groups and quandles which arise by replacing the homotopies used to give equivalence relations in the theory of fundamental groups, fundamental quandles, and knot quandles with homotopies of all orders used as arrows of categorical dimensions one and greater, and discuss other related constructions of weakly $\omega$-categorifed algebras.
\end{abstract}

% NOTE: it is good practice to \label all headings (and proclamations) immediately

\section{Introduction}

Crane's original conception of categorification \cite{C,CF,CY} was that categorical structures from which passing to the set of isomorphism classes resulted in an interesting algebraic structure governing a functorially constructed field theory should be expected to govern the behavior of functorially constructed field theories with corresponding shift in topological (or physical) dimension.  So that, for instance, (3+1)-dimensional TQFTs should be constructed from ``Hopf categories'' in manners analogous to the constructions of (2+1)-dimensional TQFTs from Hopf algebras.  Tensor categories from this point of view are then seen as categorifications of rigs.

The notion has since been greatly broadened to include what might be thought of as iterations of the original conception. For instance, a monoidal category is a categorification of a monoid in the original sense. And then monoidal bicategories (tricategories with one object), monoidal tricategories, $\ldots$, and monoidal (weak) $\omega$-categories are all higher categorificaitons of monoids.  Hence there are now as many different conceptions of categorification as there are notions of higher categories.

It was a mantra of Crane's when he first conceived the notion \cite{Cp} that, like quantization, ``categorification is not a functor.''  At the level of rigs (or more generally models of some algebraic theory) this is certainly correct, as there are multiple fusion categories with the same fusion rig.  However, as was shown in \cite{B1,B2}, in the context of \cite{L}'s definition of weak $\omega$-categories it is possible to formulate functorial categorifications {\em at the level of theories}.

The original motivation for \cite{B1,B2} was to investigate categorifications of the quandles of \cite{J}.  At an abstract level the weak $\omega$-categorification of quandles is present in \cite{B1,B2}, by applying the general machinery developed therein to a PRO whose algebras in $\Set$ are quandles. But examples of models of the weakly $\omega$-categorified theory were lacking.  It is the purpose of this paper to remedy that lack by giving details of the expected construction in which the homotopies in \cite{J}'s development of the fundamental quandle and knot quandle are used not to give an equivalence relation, but as the 1-arrows of an $\omega$-categorified quandle, with iterated homotopies between them as the higher dimensional cells.

In the following section we fix our notational conventions.  Sections 3 and 4 review the main notions and results from \cite{B1,B2}.  Section 5 discusses categories of compactly generated weak Hausdorff spaces with additional structure, the topological setting for our main results.  Sections 6 gives our main theorem.  Section 7 considers examples of the main theorem including categorified groups and quandles, while the final section discusses our primary example, weakly $\omega$-categorified quandles, and briefly considers possible next steps in their study.

\section{Notational generalities}

We will usually denote composition in the customary non-diagrammatic order, writing the composition of $f:A\rightarrow B$ and $g:B \rightarrow C$ as $g(f)$. On occasion we will use the diagrammatic order, writing the same composition as $fg$. When we do this we will draw attention to the use of diagrammatic order in the text.  Some repeated compositions may be denoted using superscripts.  We will also have occasion to use the ``padded composition operators'' introduced in \cite{Ybook}:

\begin{definition} If $f_1,\ldots f_n$ are a sequence of maps in a monoidal category $({\mathcal{C}},\otimes, I, \alpha, \rho, \lambda)$ with the property that $t(f_i)$ and $s(f_{i+1})$ are isomorphic by a composition of prolongations of instances of the associator and unitors, then
$\lceil f_1 \ldots f_n \rceil$ denotes the composition, given in diagrammatic order by
\[ \alpha_0 f_1 \alpha_1 \ldots \alpha_{n-1} f_n \alpha_n \]

\noindent where the $\alpha_i$'s are compositions of prolongations of instances of the associator and unitors, the source of $\alpha_0$ is reduced (no tensorands of the monoidal unit $I$) and completely left parenthesized, the target of $\alpha_n$ is reduces and completely right parenthesized, and the composition is well-defined.
\end{definition}

When using this convention, we place tensor exponents on the left (resp. right) to indicate a fully left (resp. right) parenthesized tensor power of an object or arrow.

Throughout, when speaking about the categories $\Glob$ and $\Coll$, of globular sets and collections respectively, we will use the following conventions: The characters $\one$ and $\emptyset$ will be used to denote the terminal and initial globular sets respectively.  The strict $\omega$-category monad on $\Glob$ will be denoted by $\mathcal{T}$. The symbol $\mysquare$ will be used to denote the composition tensor product in $\Coll$ (defined below).  The monoidal unit for $\mysquare$ is the inclusion of generators into the free strict $\omega$-category generated by $\one$, which we will denote by $I: \one \hookrightarrow \Tone$.  The terminal collection, here denoted $\id: \Tone \rightarrow \Tone$, is given by the identity globular set morphism on  $\Tone$.  The initial collection, denoted $\{\}: \emptyset \rightarrow \Tone$, is given by the unique globular set map from the initial object $\emptyset$.  When it is clear from context, we will sometimes refer to a collection $a: \mathcal{A} \rightarrow \Tone$ by simply writing the underlying globular set $\mathcal{A}$.

We will also regularly adopt the convention of distinguishing ordinary sets $X$ from globular sets $\mathcal{X}$ by using a caligraphic font for the latter when such emphasis is helpful or clarifying.

\section{Classical PROs}

Following \cite{B1,B2} we have chosen to formalize algebraic theories using PROs (whose name is short for product category), a generalization of nonsymmetric operads which allows for operations like the diagonal.  We avoid axiomatizing a symmetric group action, preferring PROs to MacLane's notion of PROPs \cite{M65} (whose name is short for product and permutation category). 
		
		\begin{definition}
			A \textit{PRO} $P$ is a strict monoidal category whose object set is isomorphic to, and in fact identified with, $\mathbb{N}$ such that the monoidal product $+:P \times P \rightarrow P$ is given by addition of natural numbers at the level of objects.
			
			A \textit{homomorphism} of PROs is a strict monoidal functor which is given by ${\mathbbm 1}_{\mathbb N}$ on the sets of objects.
		\end{definition}
		
		Given a PRO $P$, we regard morphisms in $P(n,m)$ as operations of arity $n$ and coarity $m$.  Until an algebra for a PRO is specified, the objects $\mathbb{N}$ are placeholders for the arities and coarities of these operations, so that they may be composed with one another.  Once an algebra is specified for $P$, these `slots' will be filled with copies of the underlying object of the algebra (or elements thereof for algebras in a concrete category), justifying the use of the name operations for the morphisms of the PRO.

\begin{definition}
For an object $X$ in a monoidal category $(\mathcal{C}, \otimes, I, \ldots)$, the {\em tautological PRO} on $X$,
is the PRO $Taut_{\mathcal{C}}(X)$, with
\[ Taut_{\mathcal{C}}(X)(n, m) := {\mathcal{C}}(^{\otimes n}\!X, X^{\otimes m}) \]

\noindent with composition given in diagrammatic order by $f\cdot g := \lceil f g \rceil$, and addition given by $f + g = \lceil f\otimes g \rceil$.  (Note if $\mathcal{C}$ is strict, the padded-composition operators may be omitted.)
\end{definition}

\begin{definition} If $P$ is a PRO, {\em a $P$-algebra in a monoidal category $\mathcal{C}$} is an object $X$, equipped with a PRO homomorphism $\phi: P \rightarrow Taut_{\mathcal{C}}(X)$.

In the case where $P$ is $Taut_{\mathcal{C}}(Y)$ for some object $Y$ in a monoidal category $\mathcal{C}$, we will say that $P$-algebras are {\em models for the theory of $(Y,\mathcal{C}$}) or simply the theory of $Y$ if $\mathcal{C}$ is clear from context.
\end{definition}

In the case where $\mathcal{C}$ is $\Set$ with its cartesian monoidal structure, the maps induced on hom-sets by the PRO homomorphism $\phi: P \rightarrow Taut_\Set(X)$, 
\[ \phi_{n,m}: P(n,m) \rightarrow \Set(X^n,X^m) \] 
can be curried to give a family of maps $\ulcorner \phi_{n,m} \urcorner: P(n,m)\times X^n \rightarrow X^m$, satisfying obvious relations which can be read off from the preservation of composition and the monoidal product $+$ by $\phi$.
Similarly, if $\mathcal C$ is copowered over $\Set$ the components of the PRO homomorphism can be curried to give a family of maps $\ulcorner \phi_{n,m} \urcorner: P(n,m)\cdot ^{\otimes n}\!\!X \rightarrow X^{\otimes m}$.

As an aside, one virtue of our choice of using PROs rather than PROPs or Lawvere theories is that it allows for the interpretation of theories in not-necessarily symmetric monoidal categories, and for an immediate construction of the theory of an object in a non-symmetric monoidal category.

\section{Globular Sets, $\omega$-Categories and Globular PROs}

\subsection{ Globular Sets and Collections}

We begin by recalling the notion of a globular set:   the \textit{globe category} $\mathbb{G}$ has $\mathbb{N}$ as its set of objects.  Its morphisms are generated by $s_n: n \rightarrow n+1$ and $t_n: n \rightarrow n+1$ for all $n \in \mathbb{N}$ subject to the relations $s_{n+1}(s_n) = t_{n+1}(s_n)$ and $s_{n+1}(t_n) = t_{n+1}(t_n)$.  Schematically we depict $\mathbb G$ as

\[ [0]\begin{matrix} \stackrel{s_0}{\rightarrow} \\ \stackrel{t_0}{\rightarrow} \end{matrix} [1]
\begin{matrix} \stackrel{s_1}{\rightarrow} \\ \stackrel{t_1}{\rightarrow} \end{matrix} [2]
\begin{matrix} \stackrel{s_2}{\rightarrow} \\ \stackrel{t_2}{\rightarrow} \end{matrix} [3]
\begin{matrix} \stackrel{s_3}{\rightarrow} \\ \stackrel{t_3}{\rightarrow} \end{matrix} [4]
\begin{matrix} \stackrel{s_4}{\rightarrow} \\ \stackrel{t_4}{\rightarrow} \end{matrix}\ldots
\]

		\begin{definition}
			A \emph{globular set} is a \Set-valued presheaf on $\mathbb G$. We denote the category of globular sets by \Glob.
		\end{definition}

		More concretely, a globular set $\mathcal{G} = (\{G_n\}_{n\in \mathbb{N}},\{s_{\mathcal{G}}^n\},\{t_{\mathcal{G}}^n\})$ is specified by the following data: a countable collection of sets $\{G_n\}_{n \in \mathbb{N}}$, where each $G_n$ is called the set of $n$-cells of $\mathcal{G}$, together with source and target maps $s_{\mathcal{G}} = \{s_{\mathcal{G}}^n:G_n \rightarrow G_{n-1}\}$ and $t_{\mathcal{G}} = \{t_{\mathcal{G}}^n:G_n \rightarrow G_{n-1}\}$ subject to the relations $s_{\mathcal{G}}^n(s_{\mathcal{G}}^{n+1}) = s_{\mathcal{G}}^n(t_{\mathcal{G}}^{n+1})$ and $t_{\mathcal{G}}^n(s_{\mathcal{G}}^{n+1}) = t_{\mathcal{G}}^n(t_{\mathcal{G}}^{n+1})$ in each dimension $n \in \mathbb{N}$.

		Another piece of structure needed to define both the globular operads of \cite{B} and \cite{L} and the globular PROs of \cite{B1, B2} is the free strict $\omega$-category monad $\mathcal{T}: \Glob \rightarrow \Glob$. As shown in \cite{L}, where a detailed construction of $\mathcal T$ is given, $\mathcal T$ is a cartesian monad.  Briefly, this monad takes a globular set $\mathcal{X}$ and returns the underlying globular set of the free strict $\omega$-category generated by $\mathcal{X}$.  In other words, it takes a globular set $\mathcal{X}$ and constructs the globular set $\mathcal{T}(\mathcal{X})$ whose $n$-cells are named by all possible globular pasting diagrams of the same dimension $n$ (or less, allowing for degeneracies to represent identity cells) compatibly built out of the cells in $\mathcal{X}$.  On occasion we may follow \cite{B1, B2} and refer to the names of cells as ``globular words'', regarding them as generalized words which admit concatenation in each direction corresponding to an available composition operation in the $\omega$-category structure.
		
		Note, when we use the terms `pasting diagram' and `pasting scheme', they are in the sense used in \cite{P}, though all of the pasting schemes and diagrams will have only globes as cells.

		Of particular importance is the globular set $\Tone$ generated by the terminal globular set $\one$.  Just as elements of the natural numbers $\mathbb N$, the free monoid on the one element set, indexes lengths of words, or rather the arities of operations, in ordinary algebras, the cells of $\Tone$ index globular pasting schemes which take the role of the arities for globular pasting operations.  When we wish to regard the cells of a globular set $\mathcal{X}$ as pasting operations, we equip it with a morphism $x:\mathcal{X} \rightarrow \Tone$ which specifies globular arities as cells in $\Tone$ named by pasting schemes.  

		\begin{definition}
			A \emph{collection} is a globular set $\mathcal{X}$ equipped with a globular set map $x:\mathcal{X} \rightarrow \Tone$ called the \emph{arity map}.  The category of collections $\Coll$ is the slice category $\Glob / \Tone$.
		\end{definition}

	 $\Coll$, in addition to the usual structure of a topos including cartesian and cocartsian monoidal structures, has an additional monoidal structure with respect to a composition tensor product
	 \[\mysquare : \Coll \times \Coll \rightarrow \Coll \]
	 \noindent defined by:
		
		\begin{definition}
			Let $x: \mathcal{X} \rightarrow \Tone$ and $y: \mathcal{Y} \rightarrow \Tone$ be a pair of collections.  Their composition tensor product $x \mysquare y : \mathcal{X} \mysquare \mathcal{Y} \rightarrow \Tone$ is defined by the diagram:
			$$\xymatrix{\mathcal{X} \mysquare \mathcal{Y} \pullbackmark{0,2}{2,0} \ar[rr] \ar[dd] & & \mathcal{T}(\mathcal{Y}) \ar[r]^-{\mathcal{T}(y)} \ar[dd]^{\mathcal{T}(!_{\mathcal{Y}})} & \Tsquared \ar[r]^-{\mu_{\one}} & \Tone \\
				& &  \\
				\mathcal{X} \ar[rr]^{x} & & \Tone}$$
			where $!_{\mathcal{Y}}: \mathcal{Y} \rightarrow \one$ is the unique map from $\mathcal{Y}$ to the terminal globular set.  The underlying globular set $\mathcal{X} \mysquare \mathcal{Y}$ is the pullback of $x$ and $\mathcal{T}(!_{\mathcal{Y}})$ with the arity globular set map $x \mysquare y$ defined to be the composition along the top row.
		\end{definition}
		
		As was shown in \cite{L}, the product $\mysquare$ and the object $\one \rightarrow \Tone$, together with the associators and unitors given there (or left as an exercise to the reader), give a monoidal category structure on $\Coll$.

		At the level of individual cells, $\mathcal{X} \mysquare \mathcal{Y}$ is the collection whose $k$-cells are pairs $(a,\psi)$ consisting of a $k$-cell $a \in \mathcal{X}$ and a `globular word' $\psi$ of $k$-cells from $\mathcal{Y}$ indexed by the arity of $a$.  In $\mathcal{X} \mysquare \mathcal{Y}$, the `globular letters' in the globular word $\psi \in \mathcal{T}(\mathcal{Y})$ may be compatibly `glued together' via the shape of $x(a) \in \Tone$ in the sense that each globular letter of $\psi$ is a $k$-cell whose arity shape under $y$ can replace a particular $k$-cell in the pasting scheme which names the cell $x(a) \in \Tone$. We can thus think of the cells of $\mathcal{X} \mysquare \mathcal{Y}$ as composable pairs of a cell $a$ of $\mathcal{X}$ and a `word' of cells from $\mathcal{Y}$, each of whose `letters' may be plugged into a cell of the $k$-dimensional pasting scheme which names the arity cell $x(a)$.
		
		Furthermore, the arity for a composable pair in $\mathcal{X} \mysquare \mathcal{Y}$ may be thought of as the `sum' of the arities of each letter in $\psi$ `glued together' in the shape of the pasting scheme which names $x(a)$.  More precisely, note that the map $\mathcal{T}(y)$ takes a word of cells from $\mathcal{Y}$ and returns a word of arity cells (i.e. a cell in $\mathcal{T}^2(\one)$ which is named by a pasting diagram of pasting schemes).  The component at $\one$ of the unit transformation $\mu$ for $\mathcal{T}$ then takes this globular word of arities and returns the cell in $\Tone$ which is named by the pasting scheme we would get if we strictly pasted together this diagram of schemes.  We can think of the cells in $\mathcal{T}^2(\one)$ as being named by factorizations of pasting schemes.  From this perspective, $\mu_{\one}$ essentially reduces this factorization by specifying the cell in $\Tone$ which is named by the strict pasting composition specified by the factorization.

In fact, as was shown in \cite{B1,B2} 

	\begin{theorem}
			The category $\Coll$ has a (right)closed monoidal structure with respect to the monoidal product $\mysquare$.
		\end{theorem}
		
		The internal hom for the closed structure is given by observing that $-\mysquare B$ can be expressed in terms of standard topos-theoretic functors as
		\[ - \mysquare \mathcal{B} = \Sigma_{\mu_{\one}}\Sigma_{\mathcal{T}(b)}\mathcal{T}(!_{\mathcal{B}})^* \]
		\noindent each functor in the composition, of course, admits a right adjoint giving the internal hom from $B$ as
		$$[\mathcal{B},-] = \Pi_{\mathcal{T}(!_{\mathcal{B}})}\mathcal{T}(b)^*\mu_{\one}^*$$

\subsection{Identity objects}
	
		There are four particular collections each of which is bound up with some notion of identity.  
		
		The first is the terminal collection $\id: \Tone \rightarrow \Tone$, which as for all terminal objects in categories with finite products is the unit for the cartesian monoidal structure.
		
		There is also a collection $I: \one \hookrightarrow \Tone$ whose arity map is simply the inclusion of generators.  This collection is the unit for $\mysquare$ in $\Coll$.  (see \cite{L,B1,B2})
		
		The third collection worth noting is the initial collection $\{\}: \emptyset \rightarrow \Tone$, the unit for the cocartesian monoidal structure on $\Coll$.
		
		One final collection worth noting is given by the globular set map $[id]: \one \rightarrow \Tone$.  Note that among the many cells in $\Tone$ there are the underlying globular cells of identity morphisms created when $\mathcal{T}$ produces the underlying globular set of the free strict $\omega$-category on $\one$.  Among these identities are the following special identities.  There is the underlying 1-cell of the identity on the single vertex in $\one$.  This identity map then has an identity 2-cell that sits over it.  And over this identity 2-cell there is an identity 3-cell that sits above it.  Continuing this process, we see that there is an inclusion of the terminal globular set $\one$ into $\Tone$ whose cells are exactly the iterated identities on the single 0-cell.  This sub-object can be thought of as the globular $\omega$-analogue of the additive identity $0 \in \mathbb{N}$.  
	
\subsection{Globular operads}\label{sec-glob-operads}

		\begin{definition}
			A \emph{globular operad} is a monoid $(\mathcal O, m, e)$ in $\Coll$ with respect to the monoidal product $\mysquare$.  Here $m:{\mathcal O}\mysquare {\mathcal O}\rightarrow {\mathcal O} $ and $e:\one \rightarrow {\mathcal O}$ satisfy the usual associativity and unit conditions.

	\end{definition}
	
		Given a collection $x:\mathcal{X} \rightarrow \Tone$, as always in a closed monoidal category, the endo-hom object $[{\mathcal X},{\mathcal X}]$, is a monoid in $(\Coll, \mysquare, \one \rightarrow \Tone,...)$, and thus a globular operad.  Following \cite{B1,B2} where it is discussed in some detail, we refer to it as {\em the tautological globular operad on $\mathcal X$} and denote it by $Gtaut(\mathcal{X})$.  
	
	\subsection{Algebras for a Globular Operad}
		We follow \cite{B1,B2} in defining algebras for a globular operad by identifying the category $\Glob$ with a full subcategory of $\Coll$ and using the obvious notion of an action of a monoid in a monoidal category on another object in the same category:
		
		\begin{definition}
			A collection $x:\mathcal{X} \rightarrow \Tone$ is said to be \textit{degenerate} if the arity map factors as $x = [id](!_{\mathcal{X}})$, where $[id]: \one \rightarrow \Tone$ is the map which identifies the unique copy of $\one$ in $\Tone$ consisting exclusively of iterated identities on the single 0-cell, and $!_{\mathcal{X}}:\mathcal{X} \rightarrow \one$ is the unique map from $\mathcal{X}$ to the terminal globular set $\one$.
		\end{definition}
		
		Observe that this gives rise to an obvious full inclusion of $\Glob$ into $\Coll$ sending a globular set $\mathcal X$ to the degenerate collection
		${\mathcal X}\stackrel{!}{\rightarrow}\one \stackrel{[id]}{\rightarrow}\Tone$, and any map of globular sets to the obvious commutative triangle.  We can thus make:
		
		\begin{definition}
			An \textit{algebra} $\mathcal{A}$ \textit{for a globular operad} $(\mathcal{O},\circ,e)$ is a globular set $\mathcal{A}$, thought of as a degenerate collection, together with a collection homomorphism $\omega: \mathcal{O} \mysquare \mathcal{A} \rightarrow \mathcal{A}$ which makes the diagrams
			$$\xymatrix{(\mathcal{O} \mysquare \mathcal{O}) \mysquare \mathcal{A} \ar[rr]^{\alpha_{\mathcal{O},\mathcal{O},\mathcal{A}}} \ar[d]_{\circ^{\mathcal{O}} \mysquare \id_{\mathcal{A}}} & & \mathcal{O} \mysquare (\mathcal{O} \mysquare \mathcal{A}) \ar[rr]^{\id_{\mathcal{O}} \mysquare \omega} & & \mathcal{O} \mysquare \mathcal{A} \ar[d]^{\omega} \\
				\mathcal{O} \mysquare \mathcal{A} \ar[rrrr]_{\omega} & & & & \mathcal{A} }$$
			
			$$\xymatrix{ & \mathcal{O} \mysquare \mathcal{A} \ar[ddr]^{\omega} & \\
				\\
				\one \mysquare \mathcal{A} \ar[uur]^{e \mysquare \id_{\mathcal{A}}} \ar[rr]_-{\lambda_{\mathcal{A}}} & & \mathcal{A}}$$
			commute. 
		\end{definition}
		
		We can alternatively use $Gtaut(\mathcal{X})$ to define algebras as representations of a globular operad on a degenerate collection.
		
		\begin{definition}
			Let $o:\mathcal{O} \rightarrow \Tone$ be a globular operad.  An $\mathcal{O}\textit{-module}$ is a globular operad homomorphism $f:\mathcal{O} \rightarrow Gtaut(\mathcal{A})$ for some collection $a:\mathcal{A} \rightarrow \Tone$.  An $\textit{algebra}$ for $\mathcal{O}$ is an $\mathcal{O}$-module such that the collection is degenerate.
		\end{definition}
		
As noted in \cite{B1,B2} the two definitions of algebra agree.

	\subsection{Enriched Categories over Cartesian-Duoidal Categories}
	
	The monoidal structure given by $\mysquare$ and the cartesian product on $\Coll$ endow it with the structure of a duoidal category. Recall from \cite{BM}:
		
		\begin{definition}
			A \textit{duoidal category} is a pseudomonoid in the category $\Mon\Cat_{lax}$ of monoidal categories and lax-monoidal functors. 
			To fix notation for the component part, this is a nonuple $(\mathcal{D}, \otimes, I, \odot, U, \delta, \phi, \theta, \boxplus)$ consisting of a category $\mathcal{D}$, a pair of 2-variable functors $\otimes: \mathcal{D} \times \mathcal{D} \rightarrow \mathcal{D}$ and $\odot: \mathcal{D} \times \mathcal{D} \rightarrow \mathcal{D}$, a pair of unit objects $I$ and $U$, three morphism $\delta: I \rightarrow I \odot I$, $\phi: U \otimes U \rightarrow U$, and $\theta: I \rightarrow U$ in $\mathcal{D}$, and a natural transformation
			$$\boxplus : \otimes (\odot(-,-), \odot (-,-)) \Rightarrow \odot( \otimes(-,-) \otimes(-,-))$$
			given by components
			$$\boxplus_{A,B,C,D}: [A \odot B] \otimes [C \odot D] \rightarrow [A \otimes C] \odot [B \otimes D]$$
			with $A,B,C,D \in Obj(\mathcal{C})$, specifying a lax middle-four interchange law between the product structures.  
			
		\end{definition}
		
		In particular, $\odot$ is a lax monoidal functor over $\otimes$ (by definition) from which it follows that $\otimes$ an oplax monoidal functor over $\odot$.  Moreover, $I$ a comonoid object with respect to $\odot$.  And similarly, $U$ is a monoid object with respect to $\otimes$.
		
		The duoidal categories which will actually be needed herein are of a particular form given by the following well-established result:
		
		 \begin{proposition}
			A monoidal category with finite products has a duoidal structure in which the second monoidal product is the cartesian product.
		\end{proposition}
		
		Following \cite{B1,B2} we will refer to such duoidal categories as cartesian-duoidal categories.  Cartesian-duoidal categories, together with product-preserving lax (resp. strong, strict) monoidal functors with respect to the ``first'' or non-cartesian monoidal structure form a category we denote by $\CartDuoidal_{lax}$ (resp. $\CartDuoidal$, $\CartDuoidal_{strict}$).

		\begin{definition}
			A \textit{category enriched over a cartesian-duoidal category} 
			\[(\mathcal{D}, \otimes, I, \times, \one, \delta, \phi, \theta, \boxplus)\] or simply a $\mathcal{D}$-category, is an enriched category with respect to the monoidal structure $(\mathcal{D}, \otimes, I)$.  A $\mathcal{D}$\textit{-functor} is an enriched functor between two $\mathcal{D}$-categories which is enriched with respect to the same monoidal structure $(\mathcal{D}, \otimes, I)$.  A $\mathcal{D}$\textit{-transformation} is an enriched natural transformation between two $\mathcal{D}$-functors.
		\end{definition}
		
		Note that these enriched categories as such do not use the cartesian structure.  The presence of the cartesian structure will induce a monoidal structure on the category of $\mathcal{D}$-categories. This allows for a notion of monoidal structure on a $\mathcal{D}$-category in which, while composition is enriched with respect to the first monoidal structure, the monoidal product assembles hom-objects using cartesian product.
	
		For example, fix a cartesian-duoidal category $\mathcal{D}$. We denote the 2-category of $\mathcal{D}$-categories,  $\mathcal{D}$-functors and $\mathcal{D}$-natural transformation by  $\mathcal{D}\Cat$.  The cartesian structure on $\mathcal{D}$ induces a monoidal product
		$$\oplus: \mathcal{D}\Cat \times \mathcal{D}\Cat \rightarrow \mathcal{D}\Cat$$
		of $\mathcal{D}$-categories $\mathcal{E}$ and $\mathcal{F}$ which is given as the cartesian product on the sets of objects and, for $A, B \in Obj(\mathcal{E})$ and $X, Y \in Obj(\mathcal{F})$, we have
		$$E \oplus F((A, X),(B, Y) := E(A, B) \times F(X, Y)$$
		as the hom-objects in $\mathcal{E} \oplus \mathcal{F}$.  The unit $\one_{\oplus}$ with respect to this tensor product is the trivial $\mathcal{D}$-category consisting of a single object $*$ and a single hom-object $\one_{\oplus}(*,*) := \one$. 
		
		This allows us to define the following type of $\mathcal{D}$-category.
		
		\begin{definition}
			A \emph{cartesian-duoidal enriched monoidal $\mathcal{D}$-category} $(\mathcal{M}, \diamond, \iota)$ is a pseudomonoid in the monoidal category $(\mathcal{D}\Cat, \oplus, \one_{\oplus})$, where $\mathcal{M} \in Obj(\mathcal{D}\Cat)$, the $\mathcal{D}$-functor $\diamond: \mathcal{M} \oplus \mathcal{M} \rightarrow \mathcal{M}$ is the monoidal product, and $\iota: \one_{\oplus} \rightarrow \mathcal{M}$ is the unit $\mathcal{D}$-functor such that $\diamond$ is associative and unital, with respect to $\iota$, up to $\mathcal{D}$-transformations statifying the usual monoidal coherence conditions.
		\end{definition}

	Strict cartesian-duoidal enriched monoidal categories are monoids, rather than pseudomonoids in $(\mathcal{C}\Cat, \oplus, \one_{\oplus})$ . 
	The globular PROs of \cite{B1,B2} will be example of such.
		
\subsection{Globular PROs}
		Just as classical PROs are a specific type of monoidal category, a globular PRO is a specific type of cartesian-duoidal enriched category.  %Before formally defining globular PROs we first need to ensure that $\Coll$ is cartesian-duoidal.  But since $\Coll$ is a slice category, it has a natural cartesian product structure given by taking the pullback of two collection maps.  It then follows immediately from the proposition above that since $\Coll$ has finite products it is moreover cartesian-duoidal.  This then ensures us that the category $\Coll$ has the appropriate structure for us to define globular PROs via the following construction.
		
		\begin{definition}
			A \textit{globular PRO} is a strict cartesian-duoidal enriched monoidal $\Coll$-category $(\mathcal{P}, +, O)$ such that the object set of $\mathcal{P}$ is isomorphic to $\mathbb{N}$, the bifunctor $+: \mathcal{P} \times \mathcal{P} \rightarrow \mathcal{P}$ acts as addition of natural numbers on objects, and the unit $\Coll$-functor $O: \one_{\oplus} \rightarrow \mathcal{P}$ maps $\one$ to the additive identity $0 \in \mathbb{N}$.
		\end{definition}

		\begin{definition}
			A \textit{morphism of globular PROs} between globular PROs $\mathcal{P}$ and $\mathcal{P}'$ is a strict monoidal $\Coll$-functor $(F,\widehat{F}, e): \mathcal{P} \rightarrow \mathcal{P}'$.  
		\end{definition}
		
		Together with the morphisms between them, Globular PROs form a category which we shall here denote by $\GlobPRO$.
		
		It is a theorem of \cite{B1,B2} that $\GlobPRO$ is monadic over the category $\mathbb{N}\Coll\Graph$, the full sub-category of the category of $\Coll$-enriched graphs whose objects have vertex set $\mathbb{N}$.  This is the first of a sequence of monadicity results from \cite{B1,B2}, most of which we report in \ref{mainoldstuff} below, which are proved using Kelly's notion of algebraic colimit \cite{K}.

	\subsection{The Tautological Globular PRO and Algebras}
		As was done above for globular operads, we define an algebra for a globular PRO using a $\textit{tautological globular PRO}$, which we shall denote by $GTaut(\mathcal{A})$ for a globular set $\mathcal{A}$ viewed as a degenerate collection $a:\mathcal{A} \rightarrow \Tone$.  
		
	As for all globular PROs, $GTaut(\mathcal{A})$ has as its set of objects $\mathbb{N}$. The hom-objects $GTaut(\mathcal{A})(n,m)$ are given by $[\mathcal{A}^n,\mathcal{A}^m]$, the internal hom-object of the closed structure corresponding to the product $\mysquare$ in $\Coll$.  The reader is referred to \cite{B1,B2} for a detailed exposition of the composition and monoidal product on $GTaut(\mathcal{A})$
	
	With this, we can give the succinct definition of an algebra for a globular PRO:
	
	\begin{definition}
	For a globular PRO $\mathcal{P}$ {\em a $\mathcal{P}$-algebra} is a globular set $\mathcal{A}$, viewed as a degenerate collection, equipped with a PRO homomorphism $\phi:\mathcal{A}\rightarrow GTaut(\mathcal{A})$.
	\end{definition}
	
	In \cite{B1,B2} an alternative, equivalent definition, in terms of explicit actions of the hom-objects of $\mathcal{P}$ on cartesian powers of $\mathcal{A}$ is also presented.
		
\subsection{PRO globularization and weakening \label{mainoldstuff}}

One of the main intermediate results in \cite{B1,B2} was the construction, given a PRO $P$ in $\Set$, of a globular PRO $\mathcal{P}$ whose algebras (in $\Glob$) are precisely strict $\omega$-categories equipped with a $P$-algebra structure (with operations given by strict $\omega$-functors).  

The construction is simple, though verification that the resulting structure was a PRO with the desired properties was surprisingly involved.  We do not sketch the proofs, but refer the reader to \cite{B1,B2}.

The hom-objects of $\mathcal{P}$ are given by  $\mathcal{P}(n,m) = P(n,m)\cdot \id$, the copower of the terminal collection by the corresponding hom-set of $P$.
The composition (resp. monoidal product on hom-objects) is given by the copower of the composition on $\id$ as the terminal globular operad (resp. the cartesian unitor) by set maps which reindex an iterated copower by a pair of hom-sets as a copower by their product and then apply composition in the original PRO (resp. include into another hom-set by the monoidal structure in the original PRO).  The reader should note this is {\em not} the usual construction of enriched category theory in which the hom-objects of an enrichment are constructed as copowers of the monoidal identity.  

The globular PRO whose algebras are weak $\omega$-categories with operations corresponding to those of the PRO $P$ but with all equations weakened to weak $\omega$-equivalences is constructed in \cite{B1,B2} from $\mathcal{P}$ by a process analogous to that used in \cite{L} to give the definition of weak $\omega$-categories we have adopted here.

We recall the elements and key results of the construction from \cite{B1,B2}.
	
	\begin{definition}
		Given a globular set $(\mathcal{X},s_{\mathcal{X}},t_{\mathcal{X}})$, two $n$-cells $\nu^-, \nu^+ \in \mathcal{X}$ are \textit{parallel} if $s_{\mathcal{X}}(\nu^-) = s_{\mathcal{X}}(\nu^+)$ and $t_{\mathcal{X}}(\nu^-) = t_{\mathcal{X}}(\nu^+)$.  All zero dimensional cells in $\mathcal{X}$ are parallel.
	\end{definition}
	
	Now, given a map $f:{\mathcal{X}} \rightarrow \mathcal{Y}$ of globular sets, for each $n$-cell $\nu \in \mathcal{Y}_n$ with $n > 0$ (we call such a cell an ``non-zero $n$-cell'') we may consider the set
	\begin{eqnarray*}\text{Par}_f(\nu) & := & \{(\rho^-,\rho^+) \in \mathcal{X}_{n-1} \times \mathcal{X}_{n-1} | \rho^- \text{ and } \rho^+ \text{ are parallel},\\ & & \hspace*{3em}f(\rho^-) = s_{\mathcal{Y}}^n(\nu), f(\rho^+) = t_{\mathcal{Y}}^n(\nu) \}
	\end{eqnarray*}
	of pairs of parallel $(n-1)$-cells in $\mathcal{X}$ that map via $f$ to the boundary of $\nu$ in $\mathcal{Y}$.
	
	\begin{definition}
		Given a map $f:\mathcal{X} \rightarrow \mathcal{Y}$ of globular sets, a $\textit{contraction}$ $(f:\mathcal{X} \rightarrow \mathcal{Y}, \kappa^{f})$ on $f$ is a family of maps $\kappa^{f} = \{\kappa_{\nu}:\text{Par}_f(\nu) \rightarrow \mathcal{X}_n \}$, indexed by the nonzeros $n$-cell $\nu \in \mathcal{Y}_n$, such that for each nonzero $\nu \in \mathcal{Y}$
		$$s_{\mathcal{X}}^n(\kappa_{\nu}(\rho^-,\rho^+)) = \rho^-$$
		$$t_{\mathcal{X}}^n(\kappa_{\nu}(\rho^-,\rho^+)) = \rho^+$$
		$$f(\kappa_{\nu}(\rho^-,\rho^+)) = \nu$$
		for every pair $(\rho^-,\rho^+) \in \text{Par}_f(\nu)$.
	\end{definition}

\cite{L} uses the following to give his definition of weak $\omega$-categories:	
	
	\begin{definition}
		A $\textit{contraction structure on a globular operad}$ $a: \mathcal{A} \rightarrow \Tone$ is a contraction on the unique map from $a$ to the terminal collection $\Tone \rightarrow \Tone$.  In particular, it is a contraction on the arity map $a$.
	\end{definition}
	
\cite{B1,B2} extended this to globular PROs using:

	\begin{definition}
		A $\textit{contraction structure on an}$ $\mathbb{N}\Coll\textit{-graph homomorphism}$ is a map of $\mathbb{N}\Coll$-graphs $F:\mathbf{G} \rightarrow \mathbf{H}$ such that each component $F_{n,m}:\mathbf{G}(n,m) \rightarrow \mathbf{H}(n,m)$, all of which are maps of globular sets, comes equipped with a specified contraction.
	\end{definition}

For a fixed $\mathbb{N}\Coll$-graph $\bf G$, \cite{B1,B2} showed that the category $\Cont(\sfrac{\mathbb{N}\Coll\Graph}{\mathbf{G}})$ of $\mathbb{N}\Coll$-graphs with contraction over $\bf G$, and maps of $\mathbb{N}\Coll$-graphs preserving the contraction in the obvious sense, is monadic over the slice category $\sfrac{\mathbb{N}\Coll\Graph}{\mathbf{G}}$.

	\begin{definition}
		A $\textit{contraction structure on a globular PRO homomorphism}$ is a map of globular PROs $F:\mathcal{P}^{\prime} \rightarrow \mathcal{P}$ such that its underlying map of $\mathbb{N}\Coll$-graphs is equipped with a contraction.
	\end{definition}
	
	The final monadicity result of \cite{B1,B2}, the proof of which uses both Kelly's algebraic colimits and a technical result on the monadicity of slice categories, is the following:

	\begin{theorem}
		The category $\Cont(\sfrac{\GlobPRO}{\mathcal{P}})$ of globular PROs with contraction over a fixed globular PRO $\mathcal{P}$, is monadic over $\sfrac{\mathbb{N}\Coll\Graph}{U(\mathcal{P})}$.
	\end{theorem}
	
	From this, as shown in \cite{B1,B2}, it follows that there is an initial PRO with contraction over any given globular PRO, and in particular over the PRO globularization $\mathcal{P}$ of a PRO $P$ in $\Set$.  It is this initial PRO with contraction, which we denote $\ip$ whose algebras Bressie proposes as fully weakened $\omega$-categorifications of $P$-algebras.
	
	It follows from the fact this PRO is initial among globular PROs with contraction over $\mathcal{P}$ and the fact that algebras for a globular PRO are algebras for any globular PRO which maps to it, that to construct models of the weak $\omega$-categorification of an algebraic theory, it suffices to construct algebras for {\em any} globular PRO with contraction over the globularizaiton of the PRO giving the theory.    
	
We now turn to a setting which will supply the needed globular PROs with contraction to construct the examples of weakly $\omega$-categorified groups and quandles, the latter of which were the original motivation for \cite{B1,B2}.

\section{Categories of Structured Spaces \label{spaces}}

We recall several definitions that are perhaps less well-known than they should be (cf. \cite{S}):

\begin{definition}
A topological space $X$ is {\em weakly Hausdorff} if for all $u:K\rightarrow X$ a continuous map from a compact space, $u(X)$ is closed in $X$.
\end{definition}

It is easy to see that any Hausdorff space is weakly Hausdorff.

\begin{definition}
A topological space $X$ is {\em compactly generated} if its subsets are closed if and only if their inverse images under all continuous maps from compact spaces are closed.  (Sets satisfying this latter condition in any topological space are termed {\em $k$-closed}.)
\end{definition}

It is an easy exercise to see that replacing the closed sets of a space $X$ with the sets which are $k$-closed results in the closed sets of a finer topology on the same underlying set, which is compactly generated, giving rise to a coreflection, usually denoted $k$ from the category of topological spaces to the full subcategory of compactly generated spaces, and moreover that this coreflection preserves weak Hausdorffness.

In what follows, all spaces will be compactly generated weak Hausdorff spaces.  We denote the category of compactly generated weak Hausdorff spaces by {\bf Sp}.  {\bf Sp} is, moreover, compact closed, with an internal hom ${\bf Sp}(Y,X)$ given by
$k(C^0(Y,X))$ where $C^0(Y,X)$ is given the compact-open topology.

Because compactly generated spaces are a coreflective full subcategory, they are closed under taking colimits in the category of topological spaces, and thus a CW complex is a compactly generated space (and moreover Hausdorff).  Thus the geometric realization functor from {\bf Glob} may be regarded as taking values in {\bf Sp}.

A few other observations:  both the inclusion and the coreflection restrict to the identity functor on the full subcategory of compact Hausdorff spaces, {\bf CpctHaus}.  The binary product in {\bf Sp} coincides with that the ordinary product topology if either of the spaces is locally compact, or if both are first countable (see \cite{S}).

\begin{definition}

{\em A category of structured spaces} $\mathcal{J}$ is a category enriched over {\bf Sp}, with an enriched underlying functor $U:{\mathcal{J}}\rightarrow {\bf Sp}$, with enriched copowers by compact Hausdorff spaces:
\[ {\mathcal{J}}(\Sigma\bullet X,Y) \cong {\bf Sp}(\Sigma, {\mathcal{J}}(X,Y))\]

\noindent whose underlying category has all finite colimits, and in which the copower functor $\bullet: {\bf CpctHaus}\times {\mathcal{J}} \rightarrow {\mathcal{J}}$ is left exact in both variables.

\end{definition}

\begin{definition}
For any category of structured spaces $\mathcal{J}$, the {\em homotopy category} $ho{\mathcal{J}}$ is the category with the same objects as $\mathcal{J}$ and hom sets given by $ho{\mathcal{J}}(X,Y) = \pi_0({\mathcal{J}}(X,Y))$. Its opposite category $ho{\mathcal{J}}^{op}$ has a monoidal structure induced by the binary coproduct on $\mathcal{J}$, with the initial object of $\mathcal{J}$ as its monoidal identity.
\end{definition}

Examples of categories of structured spaces include {\bf Sp}, ${\bf Sp*}$, the category of pointed (cgwH) spaces and point-preserving continuous maps, and ${\bf JP}$, the category of Joyce pairs.  We explain the structure of this last more fully:

The objects of ${\bf JP}$ are triples $(A,B,*)$ with $B$ a space, $A$ a subspace of $B$ and $*$ a point of $B\setminus A$. (Recall space means compactly generated weak Hausdorff space.) An arrow of ${\bf JP}$ from
$(A,B,*)$ to $(C,D,\dagger)$ is a continuous function $f:B\rightarrow D$ such that $f^{-1}(C) = A$ and $f(*) = \dagger$.

For a compact Hausdorff space $K$, and a Joyce pair $(A,B,*)$, the copower
$K\bullet (A,B,*)$  is $(K\times A, K\times B / (K\times\{*\}), [*])$.  Binary coproducts are given by $(A,B,*) \coprod (C,D,*) = (A\coprod C, B\vee D, [*])$, while the coequalizer of $f,g:(A,B,*)\rightarrow (C,D,*)$ is given by
\[({\rm coeq}(f|_A:A\rightarrow C,g|_A:A\rightarrow C), {\rm coeq}(f,g),[*])\] 
\noindent where the coequalizers in the components are taken in {\bf Sp}.  Note that the condition on the inverse image of the subspace implies that the first coequalizer may be canonically identified with a subspace of the second. In all three cases $[*]$ denotes the equivalence class of the original base point(s) in the quotient space.

We then have the following proposition, which gives a unified treatment of how the fundamental group and the fundamental quandle of \cite{J} arise from the fact that the circle in {\bf Sp*} (resp. the lollipop in {\bf JP}) is a homotopy cogroup (resp. homotopy coquandle). 

\begin{proposition}
For any object $X$ in a category of structured spaces $\mathcal{J}$, for each object $Y$ of $\mathcal{J}$, the hom object of the homotopy category, $ho{\mathcal{J}(X,Y)}$ is a model of the theory of $(X, ho{\mathcal{J}}^{op})$, where $ho{\mathcal{J}}^{op}$ has the monoidal structure induced by the coproduct in $\mathcal{J}$.
\end{proposition}

\begin{proof}
This follows readily from the universal property of coproducts and the fact that $\pi_0$ preserves products.  First, 
\[\textstyle ho{\mathcal{J}}(X,Y)^n = \pi_0{\mathcal{J}}(X,Y)^n \cong \pi_0({\mathcal{J}}(X,Y)^n)
\cong \pi_0({\mathcal{J}}(\coprod_n X,Y)) . \]

From this it follows that $ho{\mathcal{J}}(\coprod_m X, \coprod_n X)$ acts by $\pi_0$ of composition as operations of artity $n$ and coarity $m$, and that this operation respects both composition and the monoidal structure, so that $ho{\mathcal{J}}(X,Y)$ is a model for the theory of $(X, ho{\mathcal{J}}^{op})$.
\end{proof}

\section{Weakly $\omega$-categorified algebras}

Our main theorem will formalize the idea that, under suitable hypotheses, if the homotopies in the previous theorem are not used to give an equivalence relation, but are instead used as 1-arrows, with homotopies between them as 2-arrows, and so forth, the resulting structure will give a model of the weakly $\omega$-categorified version of the theory of $(X, ho{\mathcal{J}}^{op})$ in the sense defined in \cite{B1, B2}.  To make this precise we will need more definitions.

\begin{definition}
The {\em standard coglobular space} $D^\bullet$ is the sequence of spaces and pairs of inclusions

\[ D^0 \begin{matrix} \stackrel{s_0}{\rightarrow} \\ \stackrel{t_0}{\rightarrow} \end{matrix} D^1 
\begin{matrix} \stackrel{s_1}{\rightarrow} \\ \stackrel{t_1}{\rightarrow} \end{matrix} D^2
\begin{matrix} \stackrel{s_2}{\rightarrow} \\ \stackrel{t_2}{\rightarrow} \end{matrix} D^3
\begin{matrix} \stackrel{s_3}{\rightarrow} \\ \stackrel{t_3}{\rightarrow} \end{matrix} D^4
\begin{matrix} \stackrel{s_4}{\rightarrow} \\ \stackrel{t_4}{\rightarrow} \end{matrix}\ldots
\]

\noindent where $D^0$ is {\bf 1}, $D^n = \{ (x_1,\ldots x_n) | x_1^2 + \ldots + x_n^2 \leq 1 \} 
\subset {\mathbb R}^n$,  for $n > 0$, $s_0 = -1$, $t_0 = 1$ and $s_n(x_1,\ldots x_n) = (x_1,\ldots, x_n, 
-\sqrt{1 - x_1^2 - \ldots - x_n^2}$, $t_n(x_1,\ldots x_n) = (x_1,\ldots, x_n, \newline \sqrt{1 - x_1^2 - \ldots - x_n^2})$ for $n>0$.

\end{definition}

Observe that the necessary equations are satisfied so that $D^\bullet$ may be regarded as a covariant functor from the globe category to ${\bf CompHaus}$, which we regard as a full subcategory of {\bf Sp}, when appropriate.

The functor giving copowers by compact spaces on $\mathcal{J}$, a category of structured spaces, then gives us that for any object $X$ in $\mathcal{J}$, $D^\bullet \bullet X$ is a coglobular object in $\mathcal{J}$ (or a globular object in $\mathcal{J}^{op}$).  From this, it follows that for any object $Y$ in $\mathcal{J}$, $U({\mathcal{J}}{(D^\bullet \bullet X, Y))}$ is a globular set, where $U$ denotes the underlying set functor.

Our main theorem is then

\begin{theorem}
For $X$ any object in a category of structured spaces, if for all $n,m \in {\mathbb N}$ the path components of ${\mathcal{J}}(\coprod_n X, \coprod_m X)$ are contractible, then for any object $Y$ of $\mathcal{J}$, the globular set $\mathcal{J}{(D^\bullet \bullet X, Y)}$ admits a weak $\omega$-category structure in the sense of \cite{L}, and moreover the structure of a model of the weak $\omega$-categorification of the theory of $(X, ho{\mathcal{J}}^{op})$.
\end{theorem}

\begin{proof}
That $\mathcal{J}{(D^\bullet \bullet X, Y)}$ has a weak $\omega$-category structure is immediate once it is observed that it can be identified with the fundamental $\infty$-groupoid of ${\mathcal{J}}(X,Y)$ via the enriched adjunction defining the copower by compact Hausdorff spaces.

To see that it is a model of the weakly $\omega$-categorified theory of $(X, ho{\mathcal{J}}^{op})$ we first construct a globular PRO with contraction over $\mathcal{P}$, the strict PRO globularization of $P = Taut(X, ho{\mathcal{J}}^{op})$.
We denote this globular PRO with contraction over $\mathcal{P}$ by $\wkp$.  The proof is then completed by giving an action of $\wkp$ on the globular set 
$\mathcal{J}{(D^\bullet \bullet X, Y)}$.

To construct $\wkp$ with its globular PRO homomorphism to $\mathcal{P}$ we must first consider the structure of $\mathcal{P}$.  Its hom collections are the copowers in {\bf Coll} of the terminal collection $\id$ by the hom sets of the PRO $P$, so that $\mathcal{P}(n,m) = \pi_0({\mathcal{J}}(\coprod_m X, \coprod_n X))\bullet \id$. Recall from \cite{B1, B2} that an $n$-cell of ${\mathcal{P}}(n,m)$ lying over the cell of $\Tone$ named by a globular pasting diagram $\pi$ represents the operation in the strict $\omega$-categorification of $P$ which is equally well described as pasting $n$-tuples of cells together according to the pattern of $\pi$ in to single $n$-tuple of composite cells, then applying an operation of $P$ to get a resulting $m$-tuple of cells, or of applying the operation of $P$ to $n$-tuples of cells compatible on the pattern of $\pi$ to get $m$-tuples of cells compatible on the pattern of $\pi$, then pasting these according to the pattern of $\pi$ into an $m$-tuple of cells.  Briefly an $n$-cell of ${\mathcal{P}}(n,m)$ is named by a pair $(\pi,c)$, where $\pi$ is a globular pasting diagram and $c$ is a path component of ${\mathcal{J}}(\coprod_m X, \coprod_n X)$.

An $n$-cell of $\wkp$ lying over $(\pi,c)$ is named by a span in {\bf Sp}
\[\textstyle |\pi| \stackrel{\phi}{\longleftarrow} D^n \stackrel{\psi}{\longrightarrow} {\mathcal{J}}(\coprod_m X, \coprod_n X) \]
\noindent where $| \; |$ denotes geometric realization and in which the image of $\psi$ lies in the path component $c$, and $\phi$ preserves all (iterated) sources and targets. The source (resp. target) of the cell is named by the span obtained by replacing $|\pi|$ with $|s(\pi)|$ (resp.$|t(\pi)|$, and precomposing both $\phi$ and $\psi$ by $s_{n-1}$ (resp. $t_{n-1}$) from the standard coglobular space, noting that the left composite now has its image in the new left end of the span, by virtue of the condition that $\phi$ preserve sources and targets. 

There are a few things to notice about the cells thus named.  First, the right leg of the span naming a source or target of a cell takes values in the same path component of the mapping space as the right leg of the span naming the cell. Second, that if we fix $c$, because sources and targets are given by precomposing with maps from the standard coglobular space, the cells of $\wkp$ with right leg $c$ form a globular set, which, moreover, because sources and targets are preserved by mapping $\textstyle |\pi| \stackrel{\phi}{\longleftarrow} D^n \stackrel{\psi}{\longrightarrow} {\mathcal{J}}(\coprod_m X, \coprod_n X) $ to $\pi$ is fibered over $\Tone$, and thus has the structure of a collection.
Taking the coproduct over $c \in \pi_0({\mathcal{J}}(\coprod_m X, \coprod_n X))$ gives the collection $\wkp(n,m)$.  Finally, if the pasting diagram $\pi$ involved a $k$-fold iterated identity cell, the preservation of sources and targets required that there be a geometric $n$-cell contained in $D^n$ and intersecting each of the $2^j$ $j$-fold iterates of sources/targets of $D^n$, for $j = 1, \ldots k$, in an $(n-j)$-cell, and which is mapped onto the $(n-k)$-cell in $|\pi|$ representing the $k$-fold iterated identity, in such a way that the preimage of each point of the $k$-cell is an $(n-k)$-cell.

It now remains to show that the collections $\wkp(n,m)$ thus defined are the hom objects of a globular PRO with contraction, and that this globular PRO acts on the weak $\omega$-category $\mathcal{J}{(D^\bullet \bullet X, Y)}$.
The action depends upon the universal properties of colimits and copowers and the hypothesis that the copower functor preserves colimits in the first variable.  From these we have a sequence of isomorphism of globular sets (natural in each variable):

\[ \textstyle  \mathcal{J}{(D^\bullet \bullet X, Y)}^n \cong 
\mathcal{J}{(\coprod_n (D^\bullet \bullet X), Y)} \cong 
\mathcal{J}{(D^\bullet \bullet \coprod_n X, Y)} \] 
\[ \textstyle \cong {\bf Sp}(\{*\}, \mathcal{J}{(D^\bullet \bullet \coprod_n X, Y))} \cong {\bf Sp}(D^\bullet, \mathcal{J}{(\coprod_n X, Y))} \cong
{\bf Sp}(D^\bullet, \mathcal{J}(X, Y)^n) \]

\noindent with similar isomorphisms between sets if the standard coglobular space is replaced with some compact Hausdorff space, for instance the geometric realization of a globular pasting scheme.  By regarding a globular pasting scheme as a globular set with only finitely many cells, the preservation of finite colimits in the first variable of the copower function then has the consequence that

\[ \textstyle \mathcal{J}{(|\pi| \bullet \coprod_nX, Y)} \cong 
\int^{[m] \in {\mathbb G}} \pi([m])\times {\mathcal{J}}(D^m \bullet \coprod_n X, Y) \] 

\noindent as this particular coend can be expressed as a finite colimit.  This observation will be the starting point of the description of the action of our globular PRO on 
$\mathcal{J}{(D^\bullet \bullet X, Y)}$ and one of the ingredients for the composition in the PRO.

The other starting point for the composition is the fact that if a globular pasting scheme $\tau$ arises by compatibly coloring the globes $\gamma$ of a globular pasting scheme $\pi$ with globular pasting schemes $\tau_\gamma$, then the geometric realization $|\tau|$ can be obtained indirectly by first forming the $|\tau_\gamma|$'s and then forming the colimit of the diagram patterned on the source and target inclusions of $\pi$ in {\bf Sp}.  

Applying this last observation to the left legs and that above to the right legs it follows that if we have a coloring of a globular pasting scheme $\pi$ by cells 
\[ \textstyle |\tau_\gamma| \longleftarrow D^n \longrightarrow {\mathcal{J}}(\coprod_m X, \coprod_n X)\]
\noindent in our putative $\wkp(n,m)$ for each cell $\gamma$ of $\pi$, these can be glued together to give a span
\[ \textstyle |\tau| \longleftarrow |\pi| \longrightarrow {\mathcal{J}}(\coprod_m X, \coprod_n X).\]
\noindent And thus, from a cell $\textstyle |\pi| \stackrel{\phi}{\longleftarrow} D^n \stackrel{\psi}{\longrightarrow} {\mathcal{J}}(\coprod_m X, \coprod_n X) $ of $\wkp(n,m)$, and a coloring of the cells $\gamma$ of $\pi$ with cells 
\[ \textstyle |\tau_\gamma| \longleftarrow D^n \longrightarrow {\mathcal{J}}(\coprod_m X, \coprod_n X)\]
\noindent of $\wkp(m,p)$, we may construct the diagram

\begin{center}
    \begin{tikzpicture}
        \node (D) {$D^n$};
        \node (p) [left of=D, below of=D]{$|\pi|$};
        \node (Jpm) [right of=D, below of=D] {${\mathcal{J}}(\coprod_p X, \coprod_m X)$};
        \node (t) [left of=p, below of=p] {$|\tau|$};
        \node (Jmn) [right of=p, below of=p] {${\mathcal{J}}(\coprod_m X, \coprod_n X).$}; 
        \draw[->] (D) to (p);
        \draw[->] (D) to (Jpm);
        \draw[->] (p) to (t);
        \draw[->] (p) to (Jmn);
    \end{tikzpicture}
\end{center}

\noindent Taking the composite along the left edge as the left leg, and the map induce by the universal property of products as the right leg, we obtain a span
\[ \textstyle |\tau|\longleftarrow D^n \longrightarrow {\mathcal{J}}(\coprod_p X, \coprod_m X) \times {\mathcal{J}}(\coprod_m X, \coprod_n X). \]
Postcompsing the right leg with the composition map from $\mathcal{J}$ then gives a span
\[ \textstyle |\tau|\longleftarrow D^n \longrightarrow {\mathcal{J}}(\coprod_p X, \coprod_n X). \]

The operation thus described at the cellular level is the desired composition for the globular PRO structure on $\wkp$.  The necessary associativity follows from the associativity of iterated pasting composition and of composition in $\mathcal{J}$.

The monoidal structure, $+$, on $\wkp$ is given as follows:  given cells in $\wkp(n,m)$ and $\wkp(p,q)$, $\textstyle |\pi| \stackrel{\phi}{\longleftarrow} D^n \stackrel{\psi}{\longrightarrow} {\mathcal{J}}(\coprod_m X, \coprod_n X) $ and $\textstyle |\pi| \stackrel{\xi}{\longleftarrow} D^n \stackrel{\zeta}{\longrightarrow} {\mathcal{J}}(\coprod_p X, \coprod_q X) $ respectively (note that the product is in ${\bf Coll}$ so the pasting diagram in the left legs must be equal), their sum is then given by
\[ \textstyle |\pi| \stackrel{\phi}{\longleftarrow} D^n \stackrel{(\psi,\zeta)}{\longrightarrow} {\mathcal{J}}(\coprod_m X, \coprod_n X) \times {\mathcal{J}}(\coprod_p X, \coprod_q X) \rightarrow {\mathcal{J}}(\coprod_{m+p} X, \coprod_{n+q} X) \]
\noindent where the rightmost map is that induced by binary coproduct on $\mathcal{J}$ as a functor of two variables.

The reader may note that a choice was made, to make the left leg of the sum equal to the left leg of the first input.  This will be of no consequence once the contraction structure has been described, as $|\pi|$ is contractible and there will necessarily be an equivalence between this choice and any other, induced by lifts of iterated identity arrows on $\pi$.

The verifications that the monoidal product $+$ is functorial and associative and unital are easy, though notationally cumbersome, diagram chases at the level of cells and are left to the reader. 

We have thus far not used the hypothesis that the path components of the hom spaces of $\mathcal{J}$ are contractible.  This, however, is crucial to the existence of a contraction for the globular PRO homomorphism $\wkp \rightarrow {\mathcal{P}}$.  Recalling that cells in $\mathcal{P}$ are named by a globular pasting scheme $\pi$ and an operation of $P$, the latter of which in the present context is named by a connected component of a hom-space of $\mathcal{J}$, consider a triple of the sort for which a lift must be described to give a contraction:
\[ \left((\pi,c), |s(\pi)|\stackrel{\phi}\leftarrow D^k \stackrel{\psi}\rightarrow {\mathcal{J}}(\coprod_m X, \coprod_n X), |t(\pi)|\stackrel{\xi}\leftarrow D^k \stackrel{\zeta}\rightarrow {\mathcal{J}}(\coprod_m X, \coprod_n X)\right) \]
\noindent for $\pi$ a $k+1$-dimensional pasting scheme, in which the restrictions of corresponding maps in the spans to the source (resp. target) of $D^k$ agree, and the right legs factor through inclusion of the path component $c$.

Taking the cofibered sum over the source and target gives a span
\[ |\pi| \leftarrow S^k \rightarrow {\mathcal{J}}(\coprod_m X, \coprod_n X) \]
\noindent in which the image of the left leg is the union of the source and target, and the right leg factors through $c$.  Because both $|\pi|$ and $c$ are contractible, this extends to a span 
\[ |\pi| \leftarrow D^{k+1} \rightarrow {\mathcal{J}}(\coprod_m X, \coprod_n X) \]
\noindent and a choice of such an extension for each such triple gives the needed contraction.  (As an aside, one might make a ``uniform'' choice of the contraction by choosing an explicit contraction of each $|\pi|$ for every globular pasting scheme and of each connected component of ${\mathcal{J}}(\coprod_m X, \coprod_n X)$ for each pair $(n,m)$, and always using the fillers induced by this choice.)

The construction of the action on ${\mathcal{J}}(D^\bullet \bullet X, Y)$ is essentially the same as the construction of the right leg of the spans in the composition.  For any cell in $\wkp(n,m)$,  $\textstyle |\pi| \stackrel{\phi}{\longleftarrow} D^k \stackrel{\psi}{\longrightarrow} {\mathcal{J}}(\coprod_m X, \coprod_n X) $, given a coloring of the pasting diagram $|\pi|$ by cells of ${\mathcal{J}}(D^\bullet \bullet X, Y)^n \cong {\mathcal{J}}(D^\bullet \bullet \coprod_n X, Y)$ we have a diagram

\begin{center}
\begin{tikzpicture}
\node (D) {$D^k$};
\node (p) [left of=D, below of=D]{$|\pi|$};
\node (Jpm) [right of=D, below of=D] {${\mathcal{J}}(\coprod_m X, \coprod_n X)$};
\node (Jmn) [right of=p, below of=p] {${\mathcal{J}}(\coprod_n X, Y).$}; 
\draw[->] (D) to (p);
\draw[->] (D) to (Jpm);
\draw[->] (p) to (Jmn);
\end{tikzpicture}
\end{center}

\noindent giving a map $D^k\rightarrow {\mathcal{J}}(\coprod_n X, Y) \times {\mathcal{J}}(\coprod_m X, \coprod_n X)$ which when followed by the composition map from $\mathcal{J}$ gives the action of the cell of $\wkp(n,m)$ on ${\mathcal{J}}(D^\bullet \bullet X, Y)^n$.  Again the necessary associativity for the action follows from that for pasting composition and that for ordinary composition in $\mathcal{J}$.
\end{proof}

We will refer to algebras for $\wkp$ as constructed in the proof, in the case where $P$ is the tautological (classical) PRO on some object $X$ in $ho{\mathcal{J}}^{op}$ for some category of structured spaces, as {\em naive homotopy $\omega$-categorifictions of the theory of $X$}.

\section{Examples}
Our two primary examples to which this construction applies are those mentioned above in Section \ref{spaces}, the pointed circle $(S^1, 1)$ in {\bf Sp*} and the ``lollipop'' $N := (\{0\}, D^2\cup [1,5], 5)$ in {\bf JP} (both the circle and lollipop are described as subspaces of $\mathbb C$).  That the extra hypothesis of contractible path components in the mapping spaces holds follows in the case of $(S^1, 1)$ by virtue of the triviality of all higher homotopy groups of bouquets of circles, and in the case of the lollipop by reducing the argument to the same by first observing that any path component has representative maps with the following properties
\begin{itemize}
    \item for each copy of $N$ in the source $\coprod_m N$, the disk is mapped by $re^{i\theta} \mapsto re^{ip\theta}$ for some $p \in {\mathbb Z}$ to a disk in one of the copies of $N$ in the target $\coprod_n N$.
    \item for each copy of $N$ in the source $\coprod_m N$, the ``tail'' $[1,5]$ is mapped by a composition of paths, each of which (described as a map from $[0,1]$) is one of
    \begin{itemize}
        \item[$\diamondsuit$] $x \mapsto 4x+1$ followed by inclusion onto a copy of [1,4] in one of the copies of $N$ in the target $\coprod_n N$
        \item[$\diamondsuit$] $x \mapsto 5-4s$ followed by inclusion onto a copy of [1,4] in one of the copies of $N$ in the target $\coprod_n N$
        \item[$\diamondsuit$] $x \mapsto e^{2\pi i x}$ followed by inclusion onto a copy of $D^2$ in one of the copies of $N$ in the target $\coprod_n N$
        \item[$\diamondsuit$] $x \mapsto e^{-2\pi i x}$ followed by inclusion onto a copy of $D^2$ in one of the copies of $N$ in the target $\coprod_n N$.
    \end{itemize}
\end{itemize}

In the first of these examples, the globular set ${\bf Sp*}(D^\bullet \bullet (S^1,1), (Y,*))$ has the structure of a weak model of the theory of $((S^1,1),{\bf Sp*})$ and thus of a weak group, $(S^1,1)$ being a homotopy cogroup in ${\bf Sp*}$.  It is, unsurprisingly, essentially the same data as the one object $\infty$-groupoid equivalent to the fundamental groupoid (of the path component of $*$), but with path composition as a monoidal structure on the 0-cells named by loops at $*$ and all the categorical dimensions reduced by one.  Hardly worth all the bother we have gone through.

The second of these examples, however, is new and is in fact the original motivation for \cite{B1, B2}.  ${\bf JP}(D^\bullet \bullet N, (\Sigma, X, *))$ is a weak $\omega$-categorification of the fundamental quandle of the Joyce pair $(\Sigma, X, *)$.  In the next section we will discuss this example in more detail, constructing a weak $\omega$-categorification of the knot quandle.

It it worth mentioning in passing what is at this point a non-example.  If one considers the theory of $((S^n,*),ho{\bf Sp*}^{op})$ for $n>1$, one cannot weakly $\omega$-categorify it by the method described herein -- non-trivial elements of the homotopy groups of $S^n$ are obstructions to the construction of a contraction on the globular PRO of spans.  It is unclear at this writing whether there is a meaningful way of restricting which spans are allowed to represent $n$-cells of the categorification which would give a useful categorification of either the full theory of $((S^n,*),ho{\bf Sp*}^{op})$ or of the theory of abelian groups, for which ${\bf Sp*}(D^\bullet \bullet (S^n,*), (Y,*))$ would be a model.  

\section{Weak $\omega$-quandles}

In addition to ${\bf JP}(D^\bullet \bullet N, (\Sigma, Y, *))$ as a naive homotopy categorification of the theory of the lollipop in {\bf JP}, and thus an $\omega$-categorified quandle, we can modify the construction slightly to give other $\omega$-categorified quandles.  

First, as in \cite{J}, if $\Sigma \subset Y$ is a properly embedded oriented codimension 2 submanifold of an oriented manifold $Y$ (possibly with boundary), then ${\bf JP}(N, (\Sigma, Y, *))$ decomposes as a union of subspaces ${\bf JP}(N, (\Sigma, Y, *))_k$ each of which contains those maps from $N$ to $(\Sigma, Y, *)$ in which the linking number of $\Sigma$ and the bounding $S^1$ of the lollipop is $k$.  By the isotopy invariance of linking number, these subspaces are each necessarily a (disjoint) union of path components of the entire mapping space, from which it follows that the same decomposition may be applied to the globular space ${\bf JP}(D^\bullet \bullet N, (\Sigma, Y, *))$. 

Notice that this observation may be applied in particular to $(\Sigma, Y, *) = N$, and moreover that all of the operations which give $N$ its homotopy coquandle structure are named by the path components of ${\bf JP}(\coprod_m N, \coprod_n N)$ in which the image of the boundary of each disk in the source links the copy of $\coprod_n \{0\}$ in the target exactly once.  The PRO homomorphism from the PRO for the theory of quandles to $Taut_{\bf JP}(N)$ thus factors through this sub-PRO, which we will denote by $P_1$.  The same construction as given for tautological PROs in the main theorem, then will give a globular PRO with contraction $\wkp_1$ over ${\mathcal{P}}_1$, whose cells are named by spans of the same form as in $\wkp$, but with the right leg taking values in the path components in which every boundary circle has linking number 1 with the subspace.  The action similarly restricts to an action of $\wkp_1$ on ${\bf JP}(C^\bullet \bullet N, (\Sigma, Y, *))_1$ for any Joyce pair in which $\Sigma$ is a properly embedded submanifold of codimension 1 in a manifold $Y$.  Thus we have a weak $\omega$-categorification of the knot quandle of the pair $(\Sigma, Y, *)$.

We have, in fact, nearly constructed one more family of weakly $\omega$-categorified quandles.  Recall that the equational theory of quandles is precisely the equational theory of the two conjugation operations on groups. From this it follows that there is a PRO homomorphism from the classical PRO for quandles to that for groups.  These in turn lift to homomorphisms of the globularizations and their weakenings, and thus any model for the weak $\omega$-catetorification of groups gives rise to a model for the weak $\omega$-categorification of quandles.
In particular the globular set ${\bf Sp*}(D^\bullet \bullet (S^1,1), (Y,*))$ for any pointed space has the structure of a weak $\omega$-categorified quandle.

It is also by similar methods possible to construct $n$-categorified quandles for any $n \in \mathbb{N}$.  At the level of theories, this is accomplished either by working with presheaves on a truncated version of $\mathbb{G}$ or ``beheading'' the construction given in \cite{B1,B2} (cf. \cite{CL}).  At the level of the models constructed herein, the homotopies of $n$-times iterated homotopies must be taken as generating an equivalence relation on $n$-times iterated homotopies.  

Several avenues of research involving categorified quandles are now open: 

The first is an examination of their structure in low categorical dimensions.  Just as the defining equations of quandles correspond to the Reidemeister moves \cite{Re}, in the theory of $\lambda$-categorified quandles for any $\lambda \in ({\mathbb{N} \setminus \{0\}}) \cup \{\omega \}$ there are  1-cells which are lifts of these equations and thus correspond to implementing Reidemeister moves, and in models, 1-arrows which should in some sense represent instances of Reidemeister moves.   For $\lambda > 1$ there must also be 2-cells in the theory corresponding to the Roseman moves \cite{Ro}, the inital and final states of which correspond to different sequences of Reidemeister moves which accomplish the same overall isotopy, and for $\lambda = 1$ there must be coherence conditions corresponding to the Roseman moves on the 1-cells corresponding to the Reidemeister moves. 

We conjecture that the 1-cells corresponding to the Reidemester moves, along with those in $\ip(1,2)$ lifting the diagonal and in $\ip(2,1)$ lifting the projections, generate the 1-arrows of $\ip$ ($P$ here being the PRO presenting the theory of quandles), and moreover, in the 1-categorification of the theory, the equations on these are generated by those corresponding to the Roseman moves.   Further if $\lambda > 1$, we conjecture that the 2-cells are generated by those in $\ip(1,2)$ lifting the diagonal and in $\ip(2,1)$ lifting the projections and those corresponding to the Roseman moves.

The second, given the presence of the Roseman moves among relations (for $\lambda = 1$) or the 2-arrows (for $\lambda > 1$) in the theory of categorified quandles, is to investigate the use of categorified quandles to give invariants of knotted surfaces in $\mathbb{R}^4$.

% \section*{Acknowledgements}
% Substantial progress on this work was made while both authors were visiting the Mathematical Sciences Research Institute, Berkeley, CA and the balance while the second author was either in residence at or attached to MSRI to the extent allowed by the anti-COVID-19 social distancing regime begun in March 2020.  The authors wish to thank MSRI for its hospitality and 
% the National Science Foundation for financial support under grant DMS-0441170.

\refs

\bibitem[Batanin, 1998]{B} Batanin, M., ``Monoidal globular categories as a natural environment for the theory of weak $n$-categories,'' {\em Adv. Math.} {\bf 136}(1) (1998) 39-103.

\bibitem[Batanin \& Markl, 2012]{BM} Batanin, M. and  Markl, M., ``Centers and Homotopy Centers in Enriched Monoidal Categories,'' {\em Adv. Math.} {\bf 230} (2012) 1811-1858.

\bibitem[Bressie, 2019]{B1} Bressie, P., ``Globular PROs and weak $\omega$-categorification of algebraic theories,'' Kansas State University doctoral dissertation (2019) (retrievable from https://krex.k-state.edu/dspace/handle/2097/39788).

\bibitem[Bressie, 2020]{B2} Bressie, P.,  ``The $\omega$-categorification of algebraic theories,'' e-print \# arXiv:2006.07191 (2020).

\bibitem[Cheng \& Lauda, 2004]{CL} Cheng, E. and Lauda, A., {\em Higher Categories: an illustrated guidebook} (2004) retrievable from http://eugeniacheng.com/wp-content/uploads/2017/02/cheng-lauda-guidebook.pdf.

\bibitem[Crane, 1995]{C} Crane, L., ``Clock and category:  is quantum gravity algebraic?,'' {\em Journal of Mathematical Physics}, {\bf 36} (1995), 6180–93.

\bibitem[Crane, private communication]{Cp} Crane, L., private communication.

\bibitem[Crane \& Frenkel, 1994] {CF} Crane, L. and Frenkel, I., ``Four‐dimensional topological quantum field theory, Hopf categories, and the canonical bases,'' J. Math. Phys. {\bf 35} (10) (1994) 5136-5145.

\bibitem[Crane \& Yetter, 1998]{CY} Crane, L. and Yetter, D.N., ``Examples of Categorification,''  {\it Cahiers de Topologie et G\'{e}om\'{e}trie Diff\'{e}rentielle Cat\'{e}gorique}. {\bf 39} (1)
(1998) 3-25.

\bibitem[Joyce, 1982]{J} Joyce, D., ``A classifying invariant of knots, the knot quandle,'' {\em Journal of Pure and Applied Algebra,} {\bf 23} (1982) 37–65. 

\bibitem[Kelly, 1980]{K} Kelly, G.M., ``A unitifed treatment of transfinite constructions of free algebras, free monoids, colimits, associated sheaves, and so on,''  {\em Bull. Aust. Math. Soc.} 22(1) (1980) 1-83.

\bibitem[Leinster, 2004]{L} Leinster, T., {\em Higher Operads, Higher Categories},
London Mathematical Society Lecture Notes Series, Cambridge University Press (2004).

\bibitem[Mac Lane, 1965]{M65} Mac Lane, S., ``Categorical algebra,'' {\em Bull. Am. Math Soc.} 71(1) (1965) 40-106.

\bibitem[Power, 1991]{P} Power, J., ``An n-categorical pasting theorem,'' in {\em Category Theory} (Carboni, A., Pedicchio, M.C. and Rosolini, G, eds.) Springer (1991) 326-358.

\bibitem[Reidemeister, 1932]{Re} Reidemeister, K. {\em Knottentheorie} Springer-Verlag, Berlin, (1932).

\bibitem[Roseman, 1998]{Ro} Roseman, D., ``Reidemeister-type moves for surfaces in four-dimensional space,'' in  {\em Knot Theory, Papers from the mini-semester, Warsaw, 1995}, Banach Center Publications, vol. 42, Polish Acad. Sci. (1998)347-380

\bibitem[Strickland, undated]{S} Strickland, N., ``The Category of CGWH Spaces,'' U. Sheffield course notes, retrieved from http://neil-strickland.staff.shef.ac.uk/courses/homotopy/cgwh.pdf

\bibitem[Yetter, 2001]{Ybook} Yetter, D.N., {\em Functorial Knot Theory:  Categories of Tangles, Coherence, Categorical Deformations and Topological Invariants,} Series on Knots and Everything vol. 26, World Scientific Press (2001).

\endrefs

\end{document}